\documentclass[12pt]{amsart}

\usepackage{amsmath}
\usepackage{amsthm}
\usepackage{amssymb}
\usepackage{graphicx}
\usepackage{epsfig}

\numberwithin{equation}{section}
\numberwithin{subsection}{section}
\numberwithin{table}{section}
\numberwithin{figure}{section}

   \newtheorem{theorem}{Theorem}[section]
   \newtheorem{lemma}[theorem]{Lemma}
   
   \newtheorem{corollary}[theorem]{Corollary}
   
   \theoremstyle{remark}
   
   \theoremstyle{definition}
   \newtheorem{definition}[theorem]{Definition}

\providecommand {\norm}[1] {\lVert#1\rVert}

\newcommand{\beq}{\begin{equation}}
\newcommand{\eeq}{\end{equation}}

\newcommand{\T}{{\mathcal{T}}}

\newcommand{\E}{\mathbb{E}}
\newcommand{\M}{\mathbb{M}}
\newcommand{\Pp}{\mathbb{P}}
\newcommand{\R}{\mathbb{R}}
\newcommand{\Z}{\mathbb{Z}}

\newcommand{\vol}{{\rm vol}\,}
\newcommand{\D}{\mathcal{D}}

\numberwithin{equation}{section}
\numberwithin{figure}{section}

\begin{document}

\title[Statistical Topology]{Statistical Topology via Morse Theory Persistence and Nonparametric
Estimation}

\date{}

\author[P.Bubenik]{Peter Bubenik}
\address{Department of
Mathematics, Cleveland State University,
Cleveland, Ohio 44115-2214}
\email{p.bubenik@csuohio.edu}


\author[G.Carlson]{Gunnar Carlson}
\address{Department of Mathematics, Stanford University,
Stanford, California 94305}
\email{gunnar@math.stanford.edu}
\thanks{Support for the second author was partially funded by
DARPA, ONR, Air Force Office of Scientific Research, and NSF.}

\author[P.T. Kim]{Peter T. Kim}
\address{Department of Mathematics and Statistics, University of
Guelph, Guelph, Ontario N1G 2W1, Canada}
\email{pkim@uoguelph.ca}
\thanks{Support for the third author was
    partially funded by NSERC grant DG 46204.}

\author[Z-M. Luo]{Zhi--Ming Luo}
\address{Department of Statistics, Keimyung University, Dalseo-Gu, Daegu, 704-701, Korea}
\email{zluo@uoguelph.ca}



\begin{abstract}
In this paper we examine the use of topological methods for multivariate
statistics.  Using persistent homology from computational
algebraic topology, a random sample is used to construct estimators of
persistent homology.  This estimation procedure can then be
evaluated using the bottleneck distance between the estimated
persistent homology and the true persistent homology.  The connection
to statistics comes from the fact that when viewed as a nonparametric
regression problem, the bottleneck distance is bounded by the sup-norm
loss.  Consequently,
a sharp asymptotic minimax bound is determined under the sup--norm risk
over H\"{o}lder classes of functions for the nonparametric regression
problem on manifolds.  This provides good convergence properties for
the persistent homology estimator in terms of the expected bottleneck
distance.
\end{abstract}

\maketitle

\section{Introduction}
\label{sec.intro}
Quantitative scientists of diverse backgrounds are being asked to
apply the techniques of their specialty to data which is greater in
both size and complexity than that which has been studied previously.
Massive, multivariate data sets, for which traditional linear
methods are inadequate, pose challenges in representation,
visualization, interpretation and analysis.  A common finding is that
these massive multivariate data sets require the development of new
statistical methodology and that
these advances are dependent on increasing technical sophistication.
Two such data-analytic techniques that have recently come to the fore
are computational algebraic topology and geometric
statistics.

Commonly, one starts with data obtained from some induced
geometric structure, such as a curved submanifold of a numerical space, or, a singular
algebraic variety. The observed data is obtained as a random sample
from this space, and the objective is to statistically recover
features of the underlying space.

In computational algebraic topology, one attempts to recover qualitative
global features of the underlying data, such as connectedness,
or the number of holes, or the existence of obstructions to certain
constructions, based upon the random sample.  In other words, one
hopes to recover the underlying topology. An advantage of topology is
that it is stable under deformations and thus can potentially lead to
robust statistical procedures.
A combinatorial construction such as the alpha complex or the
\v Cech complex, see for example \cite{zomorodianCarlsson:computingPH},
converts the data into an object for which it
is possible to compute the topology. However, it is quickly apparent
that such a construction and its calculated topology depend on the
scale at which one considers the data. A multi--scale solution to this
problem is the technique of persistent homology. It quantifies the
persistence of topological features as the scale changes.  Persistent
homology is useful for visualization, feature detection and object
recognition. Applications of persistent
topology include protein structure analysis \cite{sofw:LFMPro}, gene
expression \cite{edmp:geneExpression}, and sensor
networks \cite{deSilvaGhrist:homologicalSensorNetworks}.  In a recent application
to brain image data,  a demonstration of persistent topology in discriminating
between two populations is exhibited \cite{ckb}.

In geometric statistics one uses the underlying Riemannian structure
to recover quantitative information concerning the underlying probability
distribution and functionals
thereof.  The idea is to extend statistical estimation techniques to functions
over Riemannian manifolds, utilizing the Riemannian
structure.
One then considers the magnitude of the statistical accuracy of these
estimators.  Considerable progress has been
achieved in terms of optimal estimation \cite{hen,ef1,kk3,pe1,pe2,kok,kkl}.
Other related works include \cite{rr,ru,mr,ak,bhmr}.
There is also a growing interest in function estimation over manifolds in the
learning theory literature \cite{cs,sz,bn}; see also the references cited therein.

Although computational algebraic topology and geometric statistics appear dissimilar and seem to have different objectives, it has recently been noticed that they share a commonality through statistical sampling.  In particular, a pathway between them can be established by using elements of Morse theory. This is achieved through the fact that persistent homology can be applied to Morse functions and comparisons between two Morse functions can be assessed by a metric called the bottleneck distance.  Furthermore, the bottleneck distance is bounded by the sup--norm distance between the two Morse functions on some underlying manifold.  This framework thus provides just enough structure for a statistical interpretation.  Indeed, consider a nonparametric regression problem on some manifold.  Given data in this framework one can construct a nonparametric regression function estimator such that the persistent homology associated with this estimated regression function is an estimator for the persistent homology of the true regression function, as assessed by the bottleneck distance.
Since this will be bounded by the sup-norm loss, by providing a sharp sup--norm minimax estimator of the regression function, we can effectively bound the expected bottleneck distance between the estimated persistent homology and the true persistent homology.  Consequently, by showing consistency in the sup-norm risk, we can effectively show consistency in the bottleneck risk for persistent homology which is what we will demonstrate.  Let us again emphasize that the pathway that allows us to connect computational algebraic topology with geometric statistics is Morse theory.  This is very intriguing in that a pathway between the traditional subjects of geometry and topology is also Morse theory.

We now summarize this paper.  In Section \ref{sec.top} we
will lay down the topological preliminaries needed to state
our main results.  In Section \ref{sec.geo}, we go over the
preliminaries needed for nonparametric regression on a
Riemannian manifold.  Section \ref{sec.main} states the main
results where
sharp sup-norm minimax bounds consisting of constant and rate,
and sharp sup-norm estimators are presented.  The connection to
bounding the persistent homology estimators thus ensues.
Following this in Section \ref{sec.app}, a brief discussion of
the implementation is given.
Proofs to the main results are
collected in Section \ref{sec.proofs}.  An Appendix that contains
some technical material is included for completeness.

\section{Topological Preliminaries}
\label{sec.top} \setcounter{equation}{0}
Let us assume that $\M$ is a $d-$dimensional compact Riemannian manifold
and suppose $f:\M \to \R$ is some smooth function.  Consider the sublevel set, or, lower excursion set,
\beq
\label{sublevel}
 \M_{f \leq r} := \{x
\in \M \ | \ f(x) \leq r\} = f^{-1}( (-\infty,r]).
\eeq
It is of interest to note that
for certain classes of smooth functions, the topology of $\M$ can
be approached by studying the geometry of the function.

To be more precise, for some smooth $f:\M \to \R$, consider a
point $p \in \M$ where in local coordinates the derivatives,
$\partial f /\partial x_j$ vanishes.
Then that point is called a critical point, and the evaluation $f(p)$
is called a critical value.  A critical point $p \in \M$ is called
non-degenerate if the Hessian $(\partial^2 f/\partial_i \partial_j)$ is
nonsingular.  Such functions are called Morse functions.  Later we will
see that differentiability is not needed when approached homologically.

The geometry of Morse functions can completely
characterize the homotopy type of $\M$
by the way in which topological characteristics of sublevel
sets (\ref{sublevel}) change at critical points.  Indeed
classical Morse theory tells us that the homotopy type of
(\ref{sublevel}) is characterized by attaching a cell whose
dimension is determined by the number of negative eigenvalues of
the Hessian at a critical point to the boundary of the
set (\ref{sublevel}) at the critical point.  This indeed is a
pathway that connects geometry with topology, and one in which
we shall also use to bridge statistics.  Some background material
in topology and Morse theory is provided in Appendices \ref{sec:app-1} and \ref{sec:app-2}.

As motivation
let us consider a real valued function $f$ that is a
mixture of two bump functions on the disk of radius $10$ in $\R^2$, see
Figure \ref{figure:9}.


\begin{figure}[ht]
\begin{center}
\includegraphics[width=5.5cm]{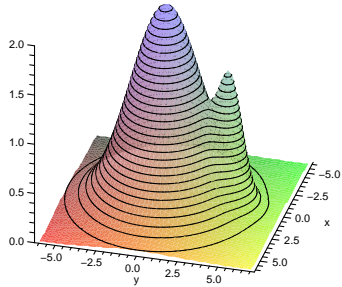}
\end{center}
\caption{A mixture of two bump functions and various contours below which are
the sublevel sets.}\label{figure:9}
\end{figure}

In this example, the maximum of $f$ equals $2$, so $\M_{f\leq 2} = \M$.
This sublevel set is the disk and therefore has no interesting topology since the disk is
contractible. In contrast, consider the sublevel sets when $r=1$,
$1.2$, and $1.5$ (see Figures \ref{figure:10}, \ref{figure:12} and \ref{figure:15}).

\begin{figure}[t]
\begin{center}
\includegraphics[width=3cm]{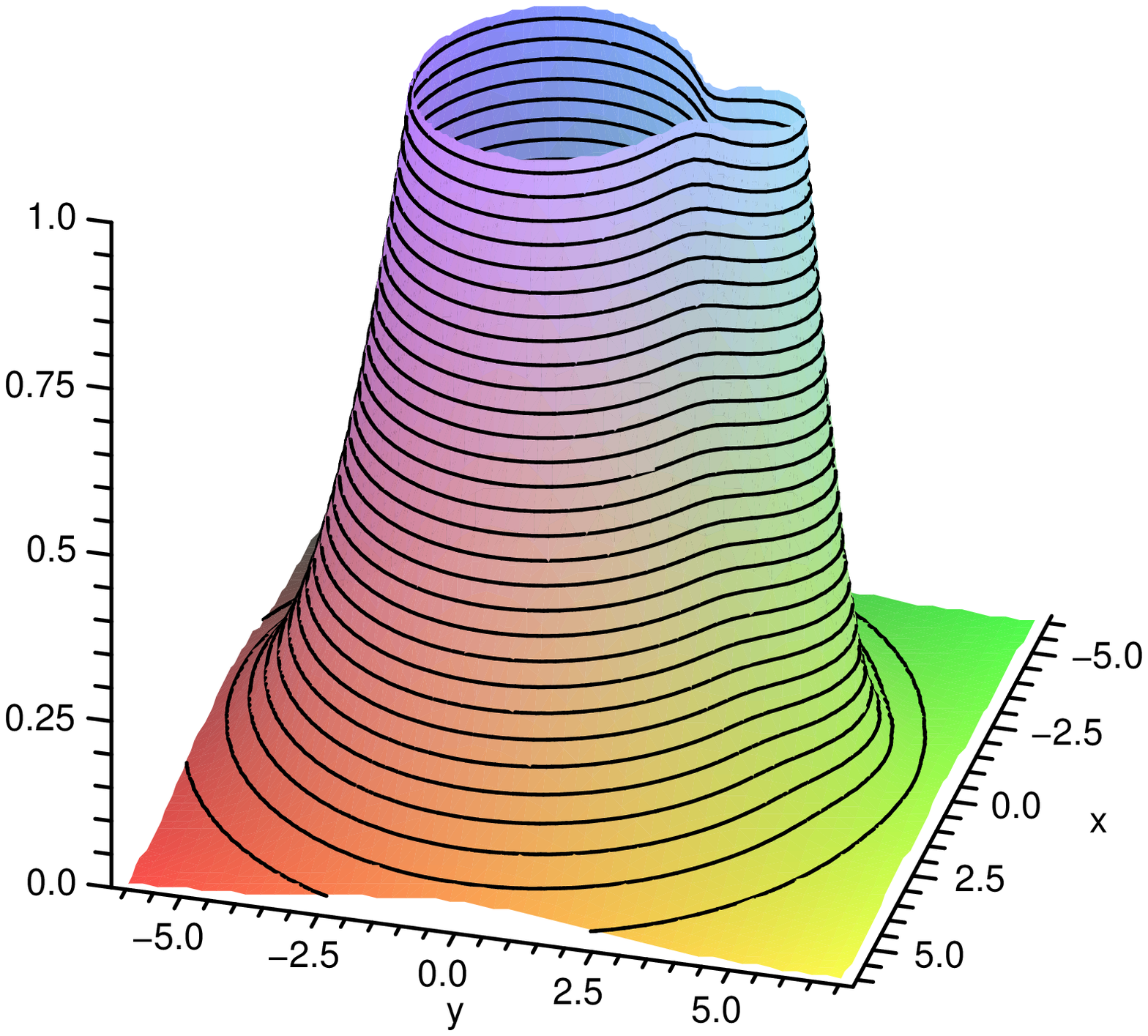}
\includegraphics[width=3cm]{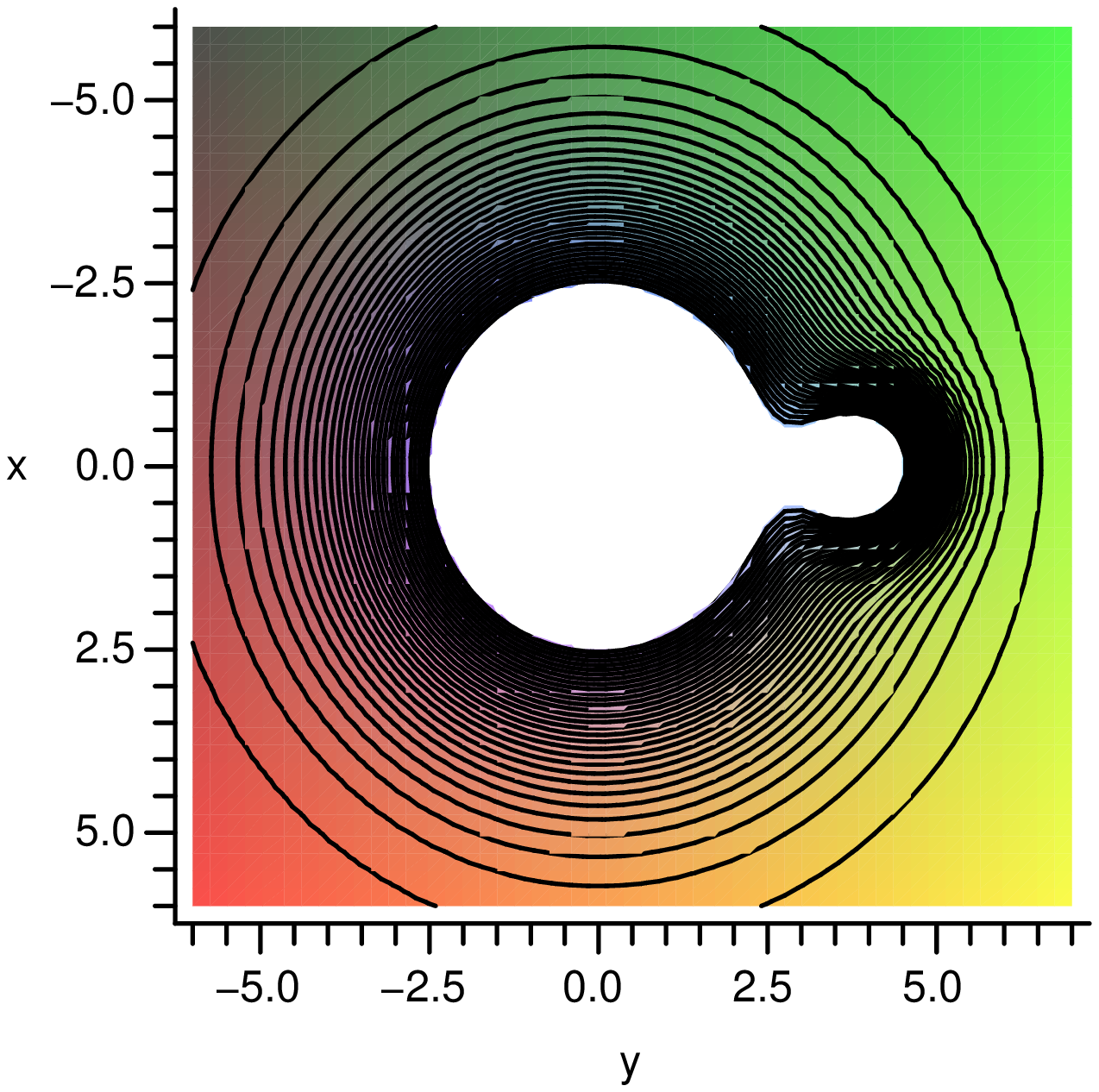}
\end{center}
\caption{The sublevel set at $r=1$ has one hole.}
\label{figure:10}
\end{figure}

\begin{figure}
\begin{center}
\includegraphics[width=3cm]{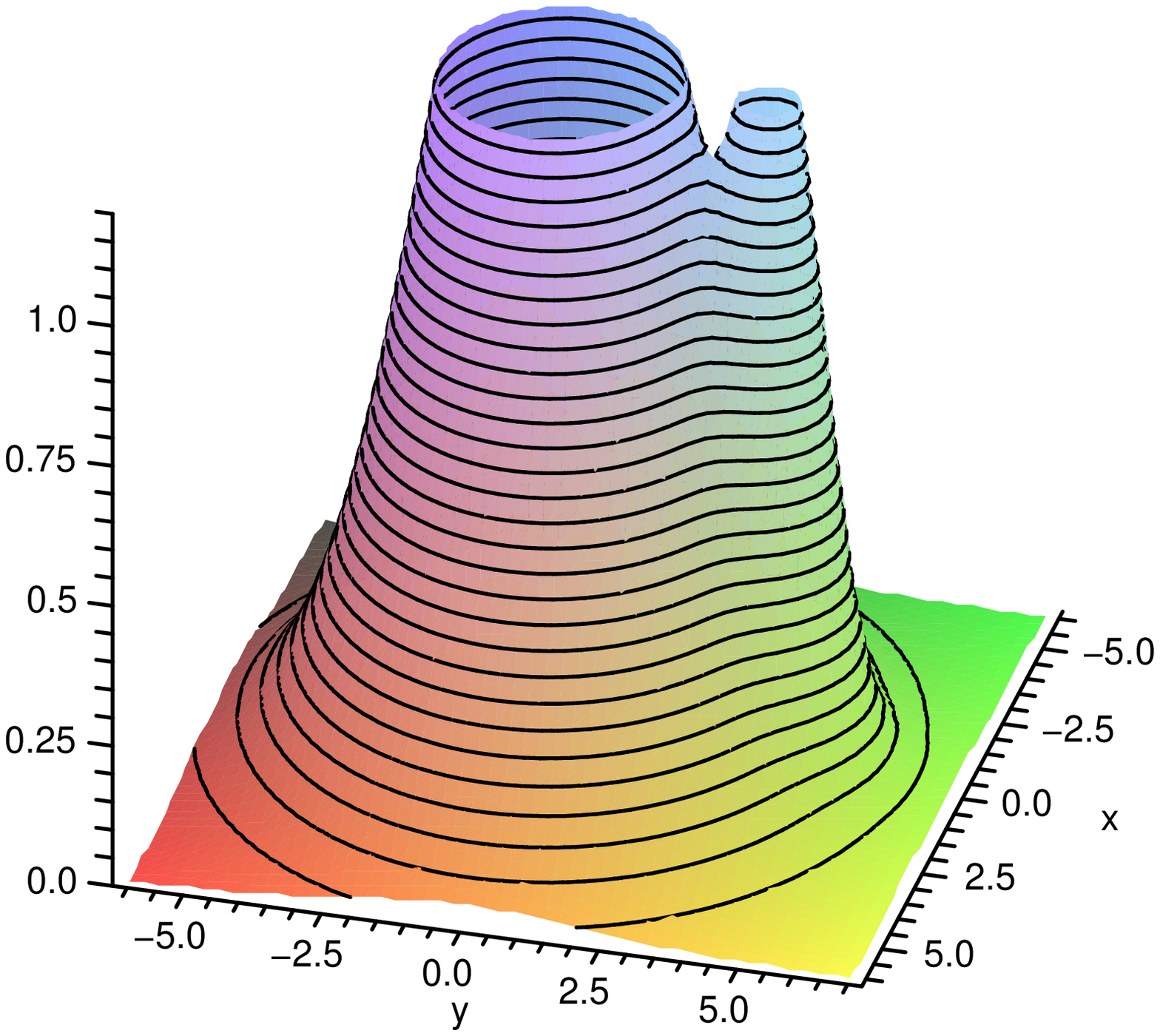}
\includegraphics[width=3cm]{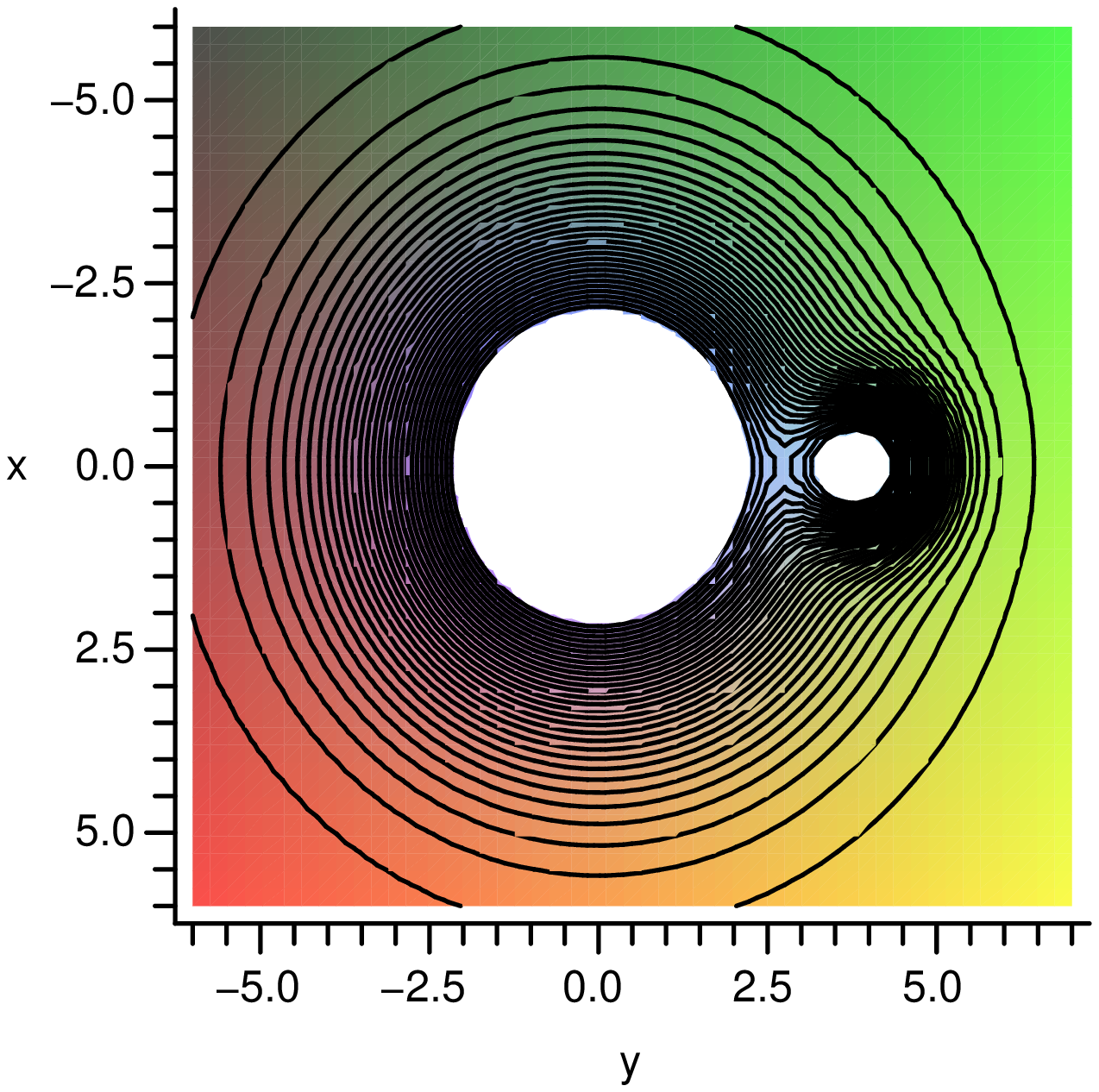}
\end{center}
\caption{The sublevel set at $r=1.2$ has two holes.}
\label{figure:12}
\end{figure}

\begin{figure}
\begin{center}
\includegraphics[width=3cm]{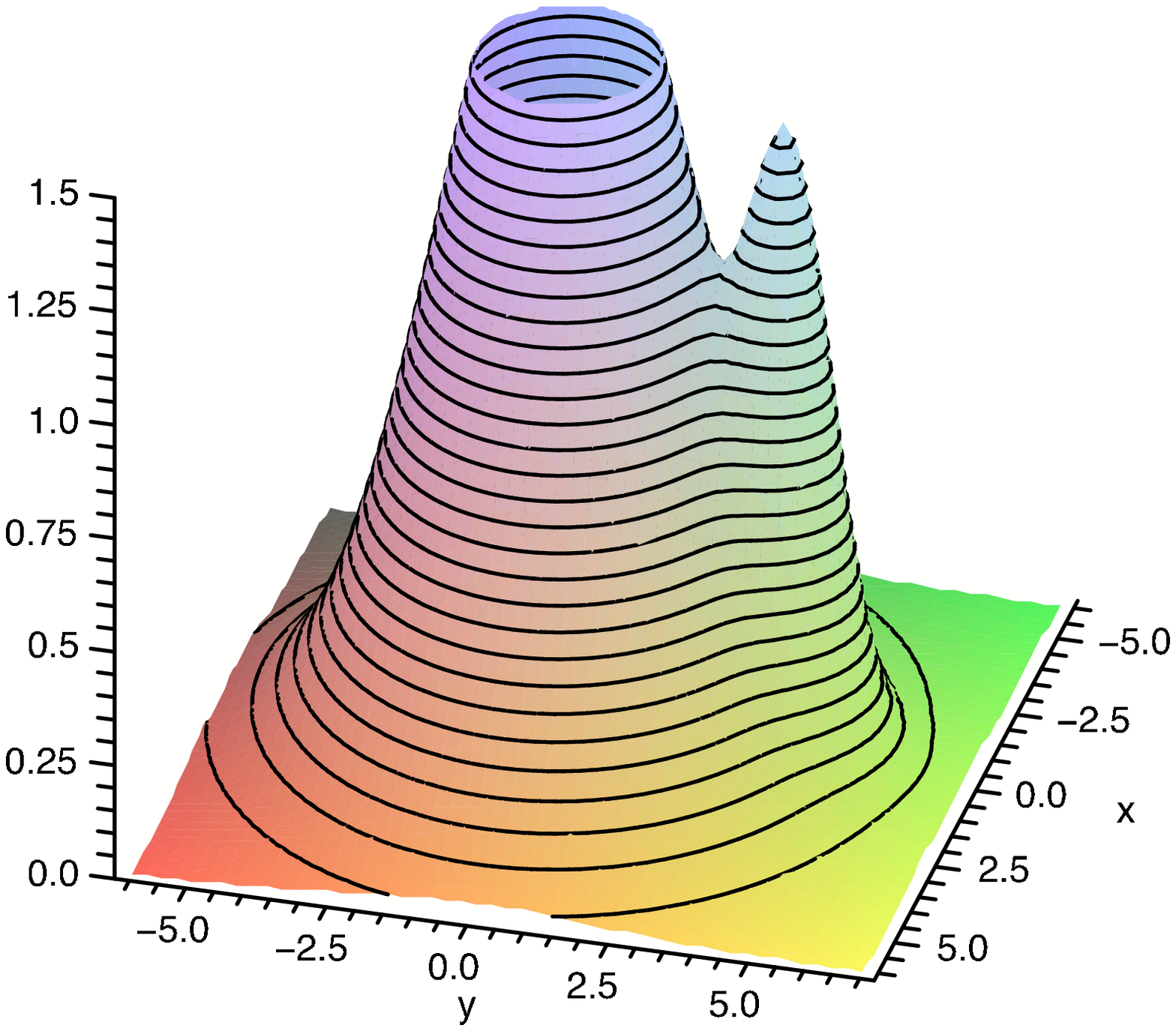}
\includegraphics[width=3cm]{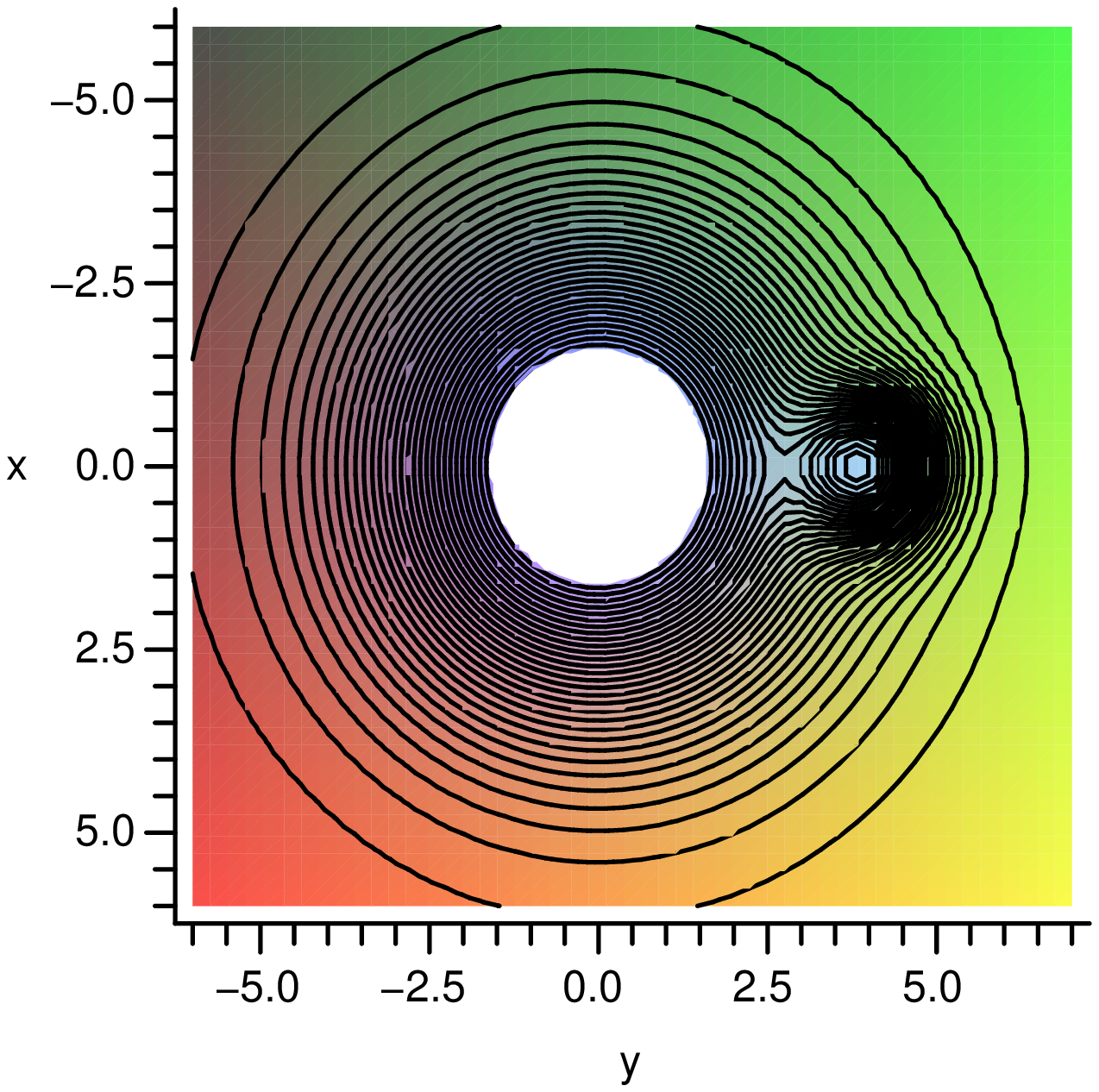}
\end{center}
\caption{The sublevel set at $r=1.5$ has one hole.}
\label{figure:15}
\end{figure}

In these cases, the sublevel sets $\M_{f\leq r}$ have non-trivial topology, namely one, two and one hole(s) respectively, each of whose boundaries is one-dimensional. This topology is detected algebraically by the first integral homology group $H_1(\M_{f\leq r})$ which will be referred to as the {homology} of degree $1$ at level $r$. This group enumerates the topologically distinct cycles in the sublevel set. In the first and third cases, for each integer $z \in \Z$, there is a cycle which wraps around the hole $z$ times. We have $H_1(\M_{f\leq r}) = \Z$.  In the second case, we have two generating non-trivial cycles and so $H_1(\M_{f\leq r}) = \Z \oplus \Z$.  For a review of homology the reader can consult Appendix \ref{sec:app-1} for related discussions.

\subsection{Persistent topology}
\label{ssec:perstop}
A computational procedure for determining how the homology persists as the level $r$ changes is provided in \cite{elz:tPaS,zomorodianCarlsson:computingPH}. 
In the above example there are two persistent homology classes (defined below). One class is born when $r=1.1$, the first sublevel set that has two holes, and dies at $r=1.4$ the first sublevel set for which the second hole disappears. The other class is born at $r=0$ and persists until $r=2$. Thus the persistent homology can be completely described by the two ordered pairs $\{ (1.1,1.4), (0,2) \}$. This is called the reduced persistence diagram (defined below) of $f$, denoted $\bar{\D}(f)$.  For a persistent homology class described by $(a,b)$, call $b-a$ its lifespan.  From the point of view of an experimentalist, a long-lived persistent homology is evidence of a significant feature in the data, while a short-lived one is likely to be an artifact.

We now give some precise definitions.

\begin{definition}
  Let $k$ be a nonnegative integer. Given $f:\M \to \R$ and $a\leq b \in \R$   the inclusion of sublevel sets $i_a^b: \M_{f\leq a} \hookrightarrow \M_{f\leq     b}$ induces a map on homology
\begin{equation*}
H_k(i_a^b): H_k(\M_{f\leq a}) \to H_k(\M_{f\leq b}).
\end{equation*}
The image of $H_k(i_a^b)$ is the persistent homology group from $a$ to $b$.
Let $\beta_a^b$ be its dimension. This counts the independent homology classes
which are born by time $a$ and die after time $b$.

  Call a real number $a$ a homological critical value of $f$ if for all
  sufficiently small $\epsilon > 0$ the map $H_k(i_{a-\epsilon}^{a+\epsilon})$
  is not an isomorphism.
Call $f$ tame if it has finitely many homological critical values, and for
each $a\in \R$, $H_k(\M_{f\leq a})$ is finite dimensional.
In particular, any Morse function on a compact manifold is tame.

Assume that $f$ is tame.
Choose $\epsilon$ smaller than the distance between any two homological critical values.
For each pair of homological critical values $a<b$, we define their multiplicity
$\mu_a^b$ which we interpret as the number of independent homology classes that
are born at $a$ and die at $b$.
We count the homology classes born by time $a+\epsilon$ that die after time
$b-\epsilon$. Among these subtract those born by $a-\epsilon$ and subtract those
that die after $b+\epsilon$. This double counts those born by $a-\epsilon$ that die
after $b+\epsilon$, so we add them back. That is,
\begin{equation*}
  \label{eq:multiplicity}
  \mu_a^b = \beta_{a+\epsilon}^{b-\epsilon} -
  \beta_{a-\epsilon}^{b-\epsilon} - \beta_{a+\epsilon}^{b+\epsilon}
  + \beta_{a-\epsilon}^{b+\epsilon}.
\end{equation*}

The persistent homology of $f$ may be encoded as follows. The reduced persistence diagram of $f$, $\bar{\D}(f)$, is the multiset
of pairs $(a,b)$ together with their multiplicities $\mu_a^b$. We call
this a diagram since it is convenient to plot these points on the plane.
We will see that it is useful to add homology classes which are born and die at
the same time. Let the persistence diagram of $f$, $\D(f)$, be given by
the union of $\bar{\D}(f)$ and $\{(a,a)\}_{a \in \R}$ where each $(a,a)$ has
infinite multiplicity.
\end{definition}

\subsection{Bottleneck distance}
\label{ssec:botdist}
Cohen--Steiner, Edelsbrunner and Harer \cite{csEH:stabilityOfPD} introduced the following  metric on the space of persistence diagrams. This metric is called the bottleneck distance and it bounds the Hausdorff distance.  It is given by
\begin{equation} \label{eq:stability}
  d_B( \D(f), \D(g) ) = \inf_{\gamma} \sup_{p \in \D(f)}
  \norm{p-\gamma(p)}_{\infty},
\end{equation}
where the infimum is taken over all bijections $\gamma: \D(f) \to
\D(g)$ and $\| \cdot \|_\infty$ denotes supremum--norm over sets.

For example, let $f$ be the function considered at the start of this section.
Let $g$ be a unimodal, radially-symmetric function on the same domain with
maximum $2.2$ at the origin and minimum $0$.
We showed that $\bar{\D}(f) = \{(1.1,1.4),(0,2)\}$. Similarly,
$\bar{\D}(g) = (0,2.2)$. The bottleneck distance is achieved by
the bijection $\gamma$ which maps $(0,2)$ to $(0,2.2)$ and $(1.1,1.4)$
to $(1.25,1.25)$ and is the identity on all `diagonal' points $(a,a)$. Since
the diagonal points have infinite multiplicity this is a bijection. Thus,
$d_B(\D(f),\D(g)) = 0.2$.

In \cite{csEH:stabilityOfPD}, the following result is proven:
\begin{equation}\label{eq:stab.inequality}
  d_B(\D(f),\D(g)) \leq \norm{f-g}_{\infty}
\end{equation}
where $f,g:\M \to \R$ are tame functions and $\| \cdot \|_\infty$ denotes sup--norm over
functions.

\subsection{Connection to Statistics}
\label{ssec:connection}
It is apparent that most articles on persistent topology do not as of yet
incorporate
statistical foundations although they do observe them heuristically.
The approach in \cite{niyogiSmaleWeinberger} combines
topology and statistics and calculates how much data
is needed to guarantee recovery of the underlying topology of the
manifold. A drawback of that technique is that it supposes that the size
of the smallest features of the data is known a priori.  To date the most
comprehensive parametric statistical approach is contained in
\cite{bubenikKim:statisticalAtPH}.  In this paper,
the unknown probability distribution is assumed to
belong to a parametric family of distributions. The data is then used to estimate
the level so as to recover the persistent topology of the underlying distribution.

As far as we are aware no statistical foundation for the nonparametric case has been formulated although \cite{csEH:stabilityOfPD} provide the topological machinery for making a concrete statistical connection.  In particular, persistent homology of a function is encoded in its reduced persistence diagram.  A metric on the space of persistence diagrams between two functions is available which bounds the Hausdorff distance and this in turn is bounded by the sup--norm distance between the two functions.  Thus by viewing one function as the parameter, while the other is viewed as its estimator, the asymptotic sup--norm risk bounds the expected Hausdorff distance thus making a formal nonparametric statistical connection.  This in turn lays down a framework for topologically classifying clusters in high dimensions.

\section{Nonparametric regression on manifolds}
\label{sec.geo}
\setcounter{equation}{0}
Consider the following nonparametric regression problem
\begin{equation}\label{wn.1}
y =  f(x) + \varepsilon, \
\  x \in \M,
\end{equation}
where $\M$ is a $d-$dimensional compact Riemannian manifold,
$f:\M \to \R$ is the regression function and
$\varepsilon$ is a normal random variable with mean zero
and variance $\sigma^2 > 0$.

For a given sample $(y_1,x_1),\ldots, (y_n,x_n)$, let $\tilde{f}$ be an
estimator of $f$ based
on the regression model (\ref{wn.1}).  We
will assess the estimator's performance by the sup--norm loss:
\begin{equation}
\label{sup-norm}
\parallel\tilde{f} - f\parallel_\infty =\sup_{x\in \M}
|\tilde{f}(x) - f(x)|.
\end{equation}
Furthermore, we will take as the parameter space, $\Lambda(\beta, L)$,
the class of H$\ddot{\mbox{o}}$lder functions
\begin{equation}
\label{holder}
\Lambda(\beta, L) = \{f: \M \to \R \mid |f(x) -f(z)|\leq L\rho(x,z)^\beta, x,z\in \M\},
\end{equation}
where $0<\beta \leq 1$ and $\rho$ is the Riemannian metric on $\M$,
i.e., $\rho(x,z)$ is the geodesic length (determined by the metric tensor)
between $x, z \in \M$.

For $w(u)$, a
continuous non-decreasing function
which increases no faster than a power of its argument as $u\to \infty$ with $w(0) = 0$,
we define the sup-norm minimax risk by
\begin{equation}
\label{minimax-risk}
r_n(w, \beta, L) = \inf_{\tilde{f}}\sup_{f\in\Lambda(\beta, L)}
\E w(\psi_n^{-1}\parallel\tilde{f} - f\parallel_\infty),
\end{equation}
where the $\psi_n \to 0$ is the sup--norm minimax rate, as $n \rightarrow \infty$, and
$\E$ denotes expectation with respect to (\ref{wn.1}) where $\varepsilon$ is normally
distributed.

\subsection{Asymptotic equidistance on manifolds}
\label{ssec:asymeqdis}
Consider a set of points $z_i \in \M$, $i=1, \cdots, m$.
We will say that the set of points is asymptotically equidistant
if
\beq\label{asymp.equid}
\inf_{i \neq j}\rho (z_i , z_j) \sim \frac{({\rm vol}\, \M)^{1/d}}{m}
\eeq
as $m \rightarrow \infty$ for all $i,j = 1, \ldots , m$, where for
two real sequences $\{a_m\}$ and $\{b_m\}$,
$a_m \sim
b_m$ will mean $|a_m/b_m|\rightarrow 1$ as $m \rightarrow \infty$, this implies
that
\beq\label{asymp.equid-1}
\frac{\max_j\min_{i\neq j} \rho (z_i , z_j)}{\min_j\min_{i\neq j}
\rho (z_i , z_j)}
\sim 1 \ ,
\eeq
as $m\rightarrow \infty$.  It will be assumed throughout that the manifold admits a
collection of asymptotically equidistant points.  This is certainly true for the
sphere (in any dimension), and will be true for all compact Riemannian manifolds
since the injectivity radius is strictly positive.  We note that \cite{pe2}
makes use of this condition as well.

We will need the following constants
\beq\label{const-1}
C_0 = L^{d/(2\beta+d)}
\left ( \frac{\sigma^2 {\rm vol}\,\M \ (\beta+d)d^2}
{{\rm vol}\, \mathbb{S}^{d-1} \beta^2} \right)^{\beta/(2\beta+d)} ,
\eeq
\beq\label{rate-1}
\psi_n = \left(\frac{\log n}{n}\right)^{\beta/(2\beta +d)} ,
\eeq
and `vol' denotes the volume of
the object in question, where $\mathbb{S}^{d-1}$ is the $(d-1)-$dimensional unit sphere with
${\rm vol}\ \mathbb{S}^{d-1}=2\pi^{d/2}/\Gamma (d/2)$ and $\Gamma$ is the gamma function.

Define the geodesic ball of radius $r > 0$ centered at $z \in \M$ by
\beq\label{geoball}
B_z(r) = \left\{ x \in \M \left| \rho(x,z) \leq r \right. \right\}.
\eeq
We have the following result whose proof will be detailed in Section \ref{ssec.upper}
\begin{lemma}
\label{lemma-aequi}
Let $z_i \in \M, i = 1, \cdots, m$, be asymptotically equidistant.
Let $\lambda = \lambda(m)$ be the largest
number such that $\bigcup^m_{i=1}\overline{B}_{z_i}(\lambda^{-1}) = \M$, where
$\overline{B}_{z_i}(\lambda^{-1})$ is the closure of the
geodesic ball of radius $\lambda^{-1}$
around $z_i$. Then there is a $C_1 >0$ such that
$
 \limsup_{m \to \infty} m\lambda(m)^{-d} \leq C_1.
$
\end{lemma}

\subsection{An estimator}
\label{ssec:estimator}
Fix a $\delta > 0$ and let
\[
m=\left[C_1\left(\frac{L(2\beta+d)}{\delta C_0 d\psi_n}\right)^{d/\beta}\right],
\]
where $C_1$ is a sufficiently large constant from Lemma \ref{lemma-aequi},
hence $m \leq n$ and
$m \to \infty$ when $n \to \infty$ and for $s \in \R$, $[s]$ denotes the greatest
integer part.

For the design points $\left\{x_i: i = 1, \ldots, n\right\}$ on $\M$,
assume that
$\left\{x_{i_j} \in \M, j=1, \ldots, m\right\}$ is
an asymptotically equidistant subset on $\M$.  Let $A_j, j = 1, \ldots, m$, be a
partition of $\M$ such that $A_j$ is the set of those $x \in \M$ for which $x_{i_j}$
is the closest point in the subset $\{x_{i_1}, \ldots, x_{i_m} \}$.  Thus, for
$j=1,\ldots , m$,
\beq\label{partition}
A_j = \left\{x\in \M \mid \rho(x_{i_j}, x) = \min_{k=1, \ldots, m}\{\rho(x_{i_k}, x)\}\right\}.
\eeq

Let $A_j$, $j=1,\ldots ,m$ be as in (\ref{partition}) and
define $\mathrm{1}_{A_j}(x)$ to be the indicator function on the set $A_j$
and consider the estimator
\beq
\label{estimator}
\hat{f}(x) = \sum^m_{j=1}\hat{a}_j \mathrm{1}_{A_j}(x),
\eeq
where for $L>0$, $0 < \beta \leq 1$,
\[
\hat{a}_j =\frac{\sum_{i=1}^n K_{\kappa, x_{i_j}}(x_{i})y_i}{\sum_{i=1}^n
K_{\kappa, x_{i_j}}(x_{i})},
\]
\[
K_{\kappa, x_{i_j}}(\omega)=\left(1-(\kappa \rho(x_{i_j}, \omega))^\beta \right)_+ ,
\]
\[
\kappa = \left( \frac{C_0\psi_n}{L}\right)^{-1/\beta} ,
\]
and $s_+=\max(s,0)$, $s\in \R$.
We remark that when $m$ is sufficiently large hence $\kappa$ is also
large, the support set of $K_{\kappa, x_{i_j}}(\omega)$ is the closed
geodesic ball $\overline{B}_{x_{i_j}}(\kappa^{-1})$ around $x_{i_j}$ for $j=1,\ldots, m$.

\section{Main Results}
\label{sec.main} \setcounter{equation}{0}
We now state the main results of this paper.
The first result provides an
upper bound for the estimator (\ref{estimator}), where the function
$w(u)$ satisfies $w(0)=0$, $w(u)=w(-u)$, $w(u)$ does not decrease, and
$w(u)$ increases not faster than a power as $u \rightarrow \infty$.

\begin{theorem}\label{thm-upper}
For the regression model {\rm (\ref{wn.1})} and the estimator {\rm (\ref{estimator})},
we have
\[
\sup_{f \in \Lambda (\beta, L)}\E w \left(\psi_n^{-1} \left\|\hat{f} -
f\right\|_\infty\right)
\leq w\left(C_0\right),
\]
as $n \to 0$,
where
$\psi_n = (n^{-1} \log  n)^{\beta/(2\beta+d)}$.
\end{theorem}

We have the asymptotic minimax result for the sup--norm risk.

\begin{theorem}
\label{thm1}
For the regression model {\rm (\ref{wn.1})}
\[
\lim_{n \to \infty}r_n(w, \beta, L) = w\left(C_0 \right).
\]
\end{theorem}
In particular, we have the immediate result.
\begin{corollary}\label{cor1}
For the regression model {\rm (\ref{wn.1})} and the estimator {\rm (\ref{estimator})},
\[
\sup_{f\in\Lambda(\beta, L)}
\E \left\| \hat{f} - f\right\|_\infty
\sim
C_0 \, \left( \frac{\log  n}n
\right)^{\beta/(2\beta+d)}
\]
as $n \to \infty$.
\end{corollary}
We note that the above generalizes earlier one-dimensional results in \cite{Ko,konu}, where the domain is
the unit interval, whereas \cite{kl} generalizes this result to
higher dimensional unit spheres.

Now that a sharp sup--norm minimax estimator has been found we would like to see how we can use this for topological data analysis.  The key is the sup--norm bound on the bottleneck distance for persistence diagrams.  In particular, for the regression function $f$ in (\ref{wn.1}) and $\hat f$ the estimator (\ref{estimator}), we have the persistence diagram $\D(f)$ as well as an estimator of the persistence diagram $\D({\hat f})$.  Using the results of Section \ref{ssec:botdist}, and in particular (\ref{eq:stab.inequality}), we have \beq\label{pd-ineq} d_B\left(\D({\hat f}),\D(f)\right) \leq \left\|{\hat f}-f\right\|_\infty . \eeq Let $\Lambda_t(\beta,L)$ denote the subset of tame functions in $\Lambda(\beta,L)$.  By corollary \ref{cor1}, the following result is immediate.

\begin{corollary}\label{cor2}
For the nonparametric regression model {\rm (\ref{wn.1})}, let
$\hat{f}$ be defined by {\rm (\ref{estimator})}.
Then
for $0 < \beta \leq 1$ and $L > 0$,
\[
\sup_{f \in \Lambda_t (\beta , L)}
\mathbb{E} d_B\left(\D(\hat{f}),\D(f)\right)
\leq
L^{d/(2\beta+d)}\left (  \frac{\sigma^2 {\rm vol}\,\M \ (\beta+d)d^2}
{{\rm vol}\, \mathbb{S}^{d-1} \beta^2} \
\frac{\log n}n \right)^{\beta/(2\beta+d)}
\]
as $n \to 0$.
\end{corollary}

\section{Discussion}
\label{sec.app}
To calculate the persistence diagrams of the sublevel sets
of $\widehat f$, we suggest that because of the way $\widehat f$ is
constructed, we can calculate its persistence diagrams using a
triangulation, $\T$ of the manifold in question.

We can then filter $\T$ using $\widehat f$ as follows.  Let $r_1 \leq r_2 \leq
\ldots \leq r_m$ be the ordered list of values of $\widehat f$ on the
vertices of the triangulation.  For $1 \leq i \leq m$, let $\T_i$ be
the subcomplex of $\T$ containing all vertices $v$ with $\widehat f(v)
\leq r_i$ and all edges whose boundaries are in $\T_i$ and all faces
whose boundaries are in $\T_i$.
We obtain the following filtration of $\T$,
\begin{equation*}
  \phi = \T_0 \subseteq \T_1 \subseteq T_2 \subseteq \cdots \subseteq \T_m = \T.
\end{equation*}
Because the critical points of $\widehat f$ only occur at the vertices of $\T$, Morse theory guarantees that the persistent homology of the sublevel sets of $\widehat f$ equals the persistent homology of the above filtration of $\T$.

Using the software Plex,~\cite{plex}, we calculate the persistent
homology, in degrees $0$, $1$, $2$, ..., $d$ of the triangulation $\T$
filtered according to the estimator.  Since
the data will be $d$--dimensional, we do not expect any interesting
homology in higher degrees, and in fact, most of the interesting features
would occur in the lower degrees.

A demonstration of this is provided in \cite{ckb} for brain image
data, where the topology of cortical thickness in an
autism study takes place.
The persistent
homology, in degrees $0$, $1$ and $2$ is calculated for 27 subjects.  Since
the data is two--dimensional, we do not expect any interesting
homology in higher degrees.
For an initial comparison of the autistic subjects and control subjects, we take the
union of the persistence diagrams, see Fig. 4 in \cite{ckb} page 392.
We note the difference in the topological structures as seen through
the persistent homologies between the
autistic and control group, particularly, as we move away from the diagonal
line.  A test using concentration pairings reveal group differences.

\section{Proofs}
\label{sec.proofs}
\setcounter{equation}{0}
Our proofs will use the ideas from \cite {kl} and \cite {Ko}.
\subsection{Upper Bound}
\label{ssec.upper}
We first prove the earlier lemma.

\begin{proof}[Proof of Lemma \ref{lemma-aequi}]
Let $(\mathcal{U}, (x^i) ) $ be any normal coordinate chart centered at $x_i$, then the
components of the metric at $x_i$ are $g_{ij} = \delta_{ij}$, so
$\sqrt{|g_{ij}(x_i)|} = 1$, see \cite {L}.  Consequently,
\begin{eqnarray*}
{\rm vol}\,(\overline{B}_{x_i}(\lambda^{-1}))&=& \int_{{B}(\lambda^{-1})}
\sqrt{|g_{ij}(\mbox{exp}_{x_i}(x))|}dx 
=\sqrt{|g_{ij}(\exp_{x_i}(t))|}
\int_{{B}(\lambda^{-1})}dx\\
&\sim&{\rm vol}\,(\mathbb{B}(\lambda^{-1}))
={\rm vol}\,(\mathbb{B}(1))\lambda^{-d}={\rm vol}\,(\mathbb{S}^{d-1})\lambda^{-d}/d \ .
\end{eqnarray*}
The first line uses the integration transformation, where
$\exp_{x_i}: {B}(\lambda^{-1}) \to \overline{B}_{x_i}(\lambda^{-1})$
is the exponential map from the tangent space $\mbox{T}\M_{x_i} \to \M$.
The second line uses the integral mean value theorem and $r$ is the radius from the
origin to point $x$ in the Euclidean ball $\mathbb{B}(\lambda^{-1})$.
The third line is asymptotic as $\lambda \to \infty$ and uses
the fact that $|g_{ij}(\mbox{exp}_{x_i}(t))| \to 1$ when $\lambda \to \infty$.
In the fourth line ${\rm vol}\,(\mathbb{B}(1))$ is the volume of $d$-dimensional
Euclidean unit
ball. The last line uses the fact ${\rm vol}\,(\mathbb{B}(1)) = {\rm vol}\,
(\mathbb{B}^{d-1})/d$.

Let $\lambda'=\lambda'(m) > 0$ be the smallest number such that
$\overline{B}_{x_i}((\lambda')^{-1})$ are disjoint. Then
$\lambda^{-1} = c(m) \times (\lambda')^{-1}$, where $c(m)>1 $ and $c(m) \to 1 $ as $m \to \infty$.  Consequently
\[
\vol(\M) \geq \sum^m_{i=1}\vol(\overline{B}_{x_i}((\lambda')^{-1}))\sim
m\vol(\mathbb{S}^{d-1})(\lambda')^{-d}/d .
\]
Thus $
 \limsup_{m \to \infty} m\lambda(m)^{-d}= \limsup_{m \to \infty} c(m)^d
 m(\lambda')^{-d}\leq \frac{d\vol(\M)}{\vol(\mathbb{S}^{d-1})}$.
\end{proof}

We now calculate the asymptotic variance of $\hat{a}_j$ for
$j=1, \ldots , m$.  Let
\[
M=\left[\frac{n{\rm vol}(\overline{B}_{x_{i_j}}(\kappa^{-1}))}{{\rm vol}(\M)}
\right].
\]
Then, 
\begin{eqnarray*}
\mbox{var}(\hat{a}_j)
&=& \frac{\sigma^2\sum_{i=1}^n
K_{\kappa, x_{i_j}}^2(x_i)}{(\sum_{i=1}^n K_{\kappa, x_{i_j}}(x_i))^2}\\
&\sim& \frac{\sigma^2{\rm vol}\,(\overline{B}_{x_{i_j}}(\kappa^{-1}))
\int_{\overline{B}_{x_{i_j}}(\kappa^{-1})}
(1-(\kappa \rho(x_{i_j}, \omega))^\beta)^2 d\omega }
{M (\int_{\overline{B}_{x_{i_j}}(\kappa^{-1})} (1-(\kappa \rho(x_{i_j}, \omega))^\beta)
d\omega)^2}\\
&=&
\frac{\sigma^2{\rm vol}\,(\overline{B}_{x_{i_j}}(\kappa^{-1}))
\int_{{B}(\kappa^{-1})} (1-(\kappa r)^\beta)^2 \sqrt{|g_{ii_j}(\exp_{x_{i_j}}(x))|}dx}
{M\ (\int_{{B}(\kappa^{-1})} (1-(\kappa r)^\beta)
\sqrt{|g_{ii_j}(\exp_{x_{i_j}}(x)))|}dx)^2}.
\end{eqnarray*}
This last expression evaluates as
$$
\frac
{
\sigma^2{\rm vol}\,(\overline{B}_{x_{j_i}}(\kappa^{-1}))
\sqrt{|g_{ii_j}(\exp_{x_{i_j}}(t))|}
\int_0^{\kappa^{-1}}\int_0^\pi\cdots\int_0^\pi\int_0^{2\pi} (1-(\kappa r)^\beta)^2
r^{d-1}dr d\sigma_{d-1}
}
{
M\ |g_{ii_j}(\exp_{x_{i_j}(t'))|}(\int_0^{\kappa^{-1}}\int_0^\pi\cdots
\int_0^{\kappa^{-1}}\int_0^\pi\cdots\int_0^\pi\int_0^{2\pi} (1-(\kappa r)^\beta)^2
r^{d-1}dr d\sigma_{d-1})^2
}
$$
so that we have
\begin{eqnarray*}
\mbox{var}(\hat{a}_j)
&\sim& \frac{\sigma^2{\rm vol}\,(\overline{B}_{x_{i_j}}(\kappa^{-1}))d{\rm vol}\,
(\mathbb{B}^d)\int_0^{\kappa^{-1}} (1-(\kappa r)^\beta)^2 r^{d-1}dr}{M \ d^2{\rm vol}\,
(\mathbb{B}^d)^2(\int_0^{\kappa^{-1}} (1-(\kappa r)^\beta)r^{d-1}dr)^2}
\\
&=& \sigma^2\kappa^d\frac{{\rm vol}\,(\M)2d(\beta+d)}{n{\rm vol}\,
(\mathbb{S}^{d-1})(2\beta+d)}
\end{eqnarray*}
as $n \to \infty$, where $d\sigma_{d-1}$ is the spherical measure on
$\mathbb{S}^{d-1}$.

\begin{lemma}
\label{Prob}
\[
\lim_{n \to \infty} \Pp \left(\psi^{-1}_n \parallel\hat{f_n}
- \E\hat{f_n}\parallel_\infty >(1+\delta)C_0\frac{2\beta}{2\beta+d}\right) = 0
\]
\end{lemma}

\begin{proof}
Denote $Z_n(x) = \hat{f_n}(x) - \E\hat{f_n}(x) $. Define
\[
D^2_n = \mbox{var}(\psi^{-1}_n Z_n(x_j)) = \psi^{-2}_n\mbox{var}(\hat{a}_j)
\sim \frac{2\beta^2C^2_0}{d(2\beta+d)\log n}.
\]
Denote $y = (1+\delta)C_02\beta/(2\beta+d)$. Then
\[
\frac{y^2}{D^2_n} =\frac{2d(1+\delta)^2\log n}{2\beta + d}.
\]
For sufficiently large $n$, $Z_n(x_j) \sim N(0, \psi^{2}_n D^2_n)$, hence
as $n \to \infty$,
\begin{eqnarray*}
\Pp\left(\parallel\psi^{-1}_n Z_n\parallel_\infty > y\right)
&\le&\Pp\left(\max_{i= 1, \cdots, m}\psi^{-1}_n | Z_n(x_j)| > y\right) \\
&\le&m \Pp\left(D^{-1}_n\psi^{-1}_n | Z_n(x_j)| > \frac{y}{D_n}\right) \\
&\le&m \exp\left\{-\frac{1}{2}\frac{y^2}{D^2_n}\right\}=m \exp\left\{-\frac{d(1+\delta)^2\log n}{2\beta + d}\right\}.
\end{eqnarray*}
Therefore
$$
\Pp\left(\parallel\psi^{-1}_n Z_n\parallel_\infty > y\right)
\le n^{-d((1+\delta)^2-1)/(2\beta+d)}(\log n)^{-d/(2\beta+d)}
D_n\left(\frac{L(2\beta+d)}{\delta C_0 d}\right)^{d/\beta}.
$$
\end{proof}

\begin{lemma}
\label{bias}
\[
\limsup_{n \to \infty} \sup_{f\in \Lambda(\beta, L)}\psi^{-1}_n \parallel f
- \E\hat{f_n}\parallel_\infty \leq (1+\delta)C_0\frac{d}{2\beta+d}
\]
\end{lemma}

\begin{proof}
We note that
\begin{eqnarray*}
\parallel f - \E\hat{f}\parallel_\infty
&=&\max_{j= 1, \ldots, m} \sup_{x\in A_j} |f(x) - \E\hat{f}(x)|\\
& \le&\max_{j= 1, \ldots, m} \sup_{x\in A_j} \left(|f(x) - f(x_j)| +
| \E\hat{f}(x_j) - f(x_j)|\right)\\
&\le&\max_{j= 1, \ldots, m} \left(| \E\hat{f}(x_j) - f(x_j)| + L
\sup_{x\in A_j} \rho(x, x_j)^\beta \right).
\end{eqnarray*}
When $m$ is sufficiently large, $A_j \subset \overline{B}_{x_j}(\lambda^{-1})$,
hence by Lemma \ref{lemma-aequi}
\[
\limsup_{n\to \infty}\sup_{x\in A_j}\rho(x, x_j) \leq \limsup_{n\to \infty}
\lambda^{-1} \leq \limsup_{n\to \infty}\left(\frac{C_1}{m}\right)^{1/d}.
\]
Thus
\[
\limsup_{n\to \infty}\sup_{x\in A_j}\psi^{-1}_n
\rho(x, x_j)^\beta \leq \limsup_{n\to \infty}\psi^{-1}_n
\left(\frac{C_1}{m}\right)^{\beta/d} \le \frac{\delta C_0 d}{L(2\beta+d)}.
\]
For $j = 1, \cdots, m$,
\begin{eqnarray*}
| \E\hat{f}(x_{i_j}) - f(x_{i_j})|
&=&| E\hat{a}_j - f(x_{i_j})| =  \left| \frac{
\sum_{j=1}^m K_{\kappa, x_i}(x_{i_j})f(x_{i_j})}{\sum_{j=1}^m
K_{\kappa, x_i}(x_{i_j})}- f(x_i)\right|\\
& \leq & \frac{\sum_{j=1}^m K_{\kappa, x_i}(x_{i_j})|f(x_{i_j})- f(x_i)|}
{\sum_{j=1}^m K_{\kappa, x_i}(x_{i_j})}\\
& \le & \frac{L\int_{\overline{B}_{x_i}(\kappa^{-1})}
(1-(\kappa \rho(x_i, \omega))^\beta) \rho(x_i, \omega))^\beta d\omega}
{\int_{\overline{B}_{x_i}(\kappa^{-1})} (1-(\kappa \rho(x_i, \omega))^\beta)d\omega}
\\
&\sim& \frac{L}{\kappa^\beta}\frac{d}{2\beta + d}=C_0 \psi_n \frac{d}{2\beta + d}
\end{eqnarray*}
as $n \to \infty$.
\end{proof}
\begin{proof}[\bf Proof of the upper bound]
\begin{eqnarray*}
&&\lim_{n \to \infty} \Pp\left(\psi^{-1}_n \parallel\hat{f}
- f\parallel_\infty >(1+\delta)C_0\right)\\
&&\le \lim_{n \to \infty} \Pp\left(\psi^{-1}_n \parallel\hat{f} - \E\hat{f}
\parallel_\infty + \psi^{-1}_n \parallel \E\hat{f} - f\parallel_\infty >(1+\delta)C_0\right)
\\
&&\le\lim_{n \to \infty} \Pp\left(\psi^{-1}_n \parallel\hat{f} -
\E\hat{f}\parallel_\infty +(1+\delta)C_0\frac{d}{2\beta+d}>(1+\delta)C_0\right)
\\
&&=\lim_{n \to \infty} \Pp\left(\psi^{-1}_n \parallel\hat{f} -
\E\hat{f}\parallel_\infty >(1+\delta)C_0\frac{2\beta}{2\beta+d}\right) = 0
\end{eqnarray*}
the second inequality uses Lemma~\ref{bias} and the last line uses Lemma~\ref{Prob}.

Let $g_n$ be the density function of $\psi^{-1}_n \parallel\hat{f} - f\parallel_\infty$, then
\begin{eqnarray*}
&&\limsup_{n \to \infty}\E w^2(\psi_n^{-1}\parallel\hat{f}_n - f\parallel_\infty)\\
&&=\limsup_{n \to \infty}\left(\int_0^{(1+\delta)C_0}w^2(x)g_n(x)dx+
\int^\infty_{(1+\delta)C_0}w^2(x)g_n(x)dx\right)\\
&&\le w^2((1+\delta)C_0) + \limsup_{n \to \infty}
\int^\infty_{(1+\delta)C_0}x^\alpha g_n(x)dx = w^2((1+\delta)C_0) \le B < \infty,
\end{eqnarray*}
where the constant $B$ does not depend on $f$, the third lines uses the
assumption on the power growth and non-decreasing property of the loss
function $w(u)$. Using the Cauchy-Schwartz inequality, we have
\begin{eqnarray*}
&&\limsup_{n \to \infty}\E
w(\psi_n^{-1}\parallel\hat{f}_n - f\parallel_\infty)
\\
&&\le w((1+\delta)C_0)\limsup_{n \to \infty}\Pp\left(\psi^{-1}_n \parallel\hat{f}
- f\parallel_\infty \le(1+\delta)C_0\right)
\\
&& + \limsup_{n \to \infty}\left\{\E w^2(\psi_n^{-1}\parallel\hat{f}_n
- f\parallel_\infty) \Pp(\psi^{-1}_n \parallel\hat{f} - f\parallel_\infty
>(1+\delta)C_0)\right\}^{1/2}
\\
&&=w((1+\delta)C_0).
\end{eqnarray*}
\end{proof}

\subsection{The lower bound}
\label{ssec.lower}
We now prove the lower bound result on $\M$.
\begin{lemma}
For sufficiently large $\kappa$, let $N=N(\kappa)$ be such that $N \to \infty$
when $\kappa \to \infty$ and $x_i \in \M, i = 1, \cdots, N$, be such that
$x_i$ are asymptotically equidistant,and such that $\overline{B}_{x_i}
(\kappa^{-1})$ are disjoint. There is a constant $0 < D < \infty$ such that
\beq \label{con2}
 \liminf_{\kappa \to \infty} N(\kappa)\kappa^{-d} \geq D.
\eeq
\end{lemma}
 \begin{proof}
 Let $\kappa' > 0$ be the largest number such that
 $\bigcup^{N(\kappa)}_{i=1}\overline{B}_{x_i}((\kappa')^{-1}) = \M$. Then
\[
(\kappa')^{-1} = c(\kappa) \times \kappa^{-1}
\]
where $c(\kappa)>1 $ and $c(\kappa) \to \mbox{const.} \ge 1 $ as $\kappa \to \infty$.
\[
{\rm vol}\,(\M) \leq \sum^N_{i=1}{\rm vol}\,(\overline{B}_{x_i}((\kappa')^{-1}))
\sim N{\rm vol}\,(\mathbf{S}^{d-1})(\kappa')^{-d}/d
\]
Thus
\[
 \liminf_{\kappa \to \infty} N(\kappa)\kappa^{-d}= \liminf_{\kappa \to \infty}
 c(\kappa)^{-d}N(\kappa')^{-d}\geq \mbox{const.} \times \frac{d{\rm vol}\,(\M)}{{\rm vol}\,(\mathbf{S}^{d-1})}.
\]
 \end{proof}

Let $J_{\kappa, x}: \M \to \mathbb{R}$, and
\[
J_{\kappa, x} = L\kappa^{-\beta}K_{\kappa, x}(x) = L\kappa^{-\beta}
(1-(\kappa d(x, x))^\beta)_+,
\]
where $\kappa>0, x \in \M$. Let $N=N(\kappa)$ be the greatest integer such
that there exists observations $x_i \in \M, i = 1, \cdots, N $ (with possible
relabeling) in the
observation set $\{x_i, i = 1, \cdots, n\}$ such that the functions
$J_{\kappa, x_i}$ have disjoint supports. From (\ref{con2})
\[
\liminf_{\kappa \to \infty} N(\kappa)\kappa^{-d} \geq \mbox{const}.
\]
Let
\[
\mathcal{C}(\kappa,\{x_i\})=\left\{\sum_{i=1}^N\theta_iJ_{\kappa, x_i}:
|\theta_i|\leq 1, i = 1, \cdots, N \right\},
\]
where $\mathcal{C}(\kappa,\{x_i\}) \subset\Lambda(\beta,L)$ when $0< \beta \leq 1$.
The complete class of estimators for estimating $f \in
\mathcal{C}(\kappa,\{x_i\})$ consists of all of the form
\beq \label{est}
\hat{f}_n = \sum_{i=1}^N\hat{\theta}_iJ_{\kappa, x_i}
\eeq
where $\hat{\theta}_i = \delta_i(z_1, \cdots, z_N), i = 1, \cdots, N$, and
\[
z_i=\frac{\sum_{j=1}^n J_{\kappa, x_i}(x_j)y_j}{\sum_{j=1}^n
J^2_{\kappa, x_i}(x_j)}.
\]
When $\hat{f}_n$ is of the form (\ref{est}) and $f \in \mathcal{C}(\kappa,\{x_i\})$
then
\begin{eqnarray*}
\parallel \hat{f}_n -f \parallel_\infty& \geq& \max_{i=1, \cdots, N}
|\hat{f}_n(x_i) -f(x_i)|= |J_{\kappa, x_1}(x_1)| \parallel \hat{\theta} -
\theta \parallel_\infty\\ &=& L\kappa^{-\beta}\parallel \hat{\theta} - \theta \parallel_\infty
\end{eqnarray*}
Hence
\begin{eqnarray*}
r_n &\geq & \inf_{\hat{f}_n}\sup_{f\in\mathcal{C}(\kappa,\{x_i\})}
\E w(\psi_n^{-1}\parallel\hat{f}_n - f\parallel_\infty)\\
&\geq & \inf_{\hat{\theta}}\sup_{|\theta_i|\leq 1}
\E w(\psi_\varepsilon^{-1}L\kappa^{-\beta}
\parallel\hat{\theta} - \theta \parallel_\infty),
\end{eqnarray*}
where the expectation is with respect to a multivariate normal
distribution with mean
vector $\theta$ and the variance-covariance matrix $\sigma_N^2{\bf I}_N$, where
${\bf I}_N$ is the $N \times N$ identity matrix and $\sigma^2_N = \mbox{var}(z_1)
= \sigma^2 /\sum_{j=1}^N J^2_{\kappa, x_i}(x_j)$.

Fix a small number $\delta$ such that $0< \delta < 2$ and
\[
C'_0 = L^{d/(2\beta+d)}\left ( \frac{(2-\delta){\rm vol}\,(\M)(\beta+d)d^2}
{2{\rm vol}\,(\mathbf{S}^{d-1})\beta^2}\right)^{\beta/(2\beta+d)}
\]
and
\[
\kappa=\left(\frac{C_0'\psi_\varepsilon}{L}\right)^{-1/\beta}.
\]
Since
\begin{eqnarray*}
\sigma^{-1}_N &=&\sigma^{-1}\sqrt{\sum_{j=1}^N J^2_{\kappa, x_i}(x_j)}\sim  \sqrt{ \frac{(2-\delta)d}{2\beta + d}\log n}\\
&\leq& \sqrt{ (2-\delta)(\log(\log n/n)^{-d/(2\beta+d)})}\\
&=&\sqrt{2-\delta}\sqrt{\log(\mbox{cons}\times \kappa^d)}=\sqrt{2-\delta}\sqrt{\log N}
\end{eqnarray*}
by (\ref{con2}),
it follows that if
\[
\sigma^{-1}_N \leq \sqrt{2-\delta}\sqrt{\log N}
\]
for some $0<\delta<2$, then
\[
\inf_{\hat{\theta}}\sup_{|\theta_i|\leq 1}
\E
w(\parallel\hat{\theta} - \theta
\parallel_\infty) \to w(1),
\]
as $N\rightarrow \infty$,
but
\[
\psi_n^{-1}L\kappa^{-\beta} = C_0'.
\]
By the continuity of the function $w$, we have
\[
\inf_{\hat{\theta}}\sup_{|\theta_i|\leq 1}
\E w(\psi_n^{-1}L\kappa^{-\beta}\parallel
\hat{\theta} - \theta \parallel_\infty) \to w(C_0'),
\]
when $N \to \infty$.
Since $\delta$ was chosen arbitrarily, the result follows.

\bigskip

\appendix

\section{Background on Topology}
\label{sec:app-1}
In this appendix we present a technical overview of homology as used in
our procedures. For an intensive treatment we refer the reader to the
excellent text \cite{spanier}.

 Homology is an algebraic procedure for counting holes in
 topological spaces. There are numerous variants of homology: we use
 simplicial homology with $\Z$ coefficients.
 Given a set of points $V$, a $k$-simplex is an unordered subset
 $\{v_0, v_1, \ldots, v_k\}$ where $v_i \in V$ and $v_i \neq v_j$ for
 all $i \neq j$. The faces of this $k$-simplex consist of all
 $(k-1)$-simplices of the form $\{v_0, \ldots, v_{i-1}, v_{i+1},
 \ldots, v_k\}$ for some $0\leq i \leq k$. Geometrically, the
 $k$-simplex can be described as follows: given $k+1$ points in $\R^m$
 ($m\geq k$), the $k$-simplex is a convex body bounded by the union of
 $(k - 1)$ linear subspaces of $\R^m$ of defined by all possible
 collections of $k$ points (chosen out of $k+1$ points).  A simplicial
 complex is a collection of simplices which is closed with respect to
 inclusion of faces.  Triangulated surfaces form a concrete example,
 where the vertices of the triangulation correspond to $V$. The
 orderings of the vertices correspond to an orientation.  Any abstract
 simplicial complex on a (finite) set of points $V$ has a geometric
 realization in some $\R^m$.  Let ${X}$ denote a simplicial
 complex. Roughly speaking, the homology of $X$, denoted $H_\ast(X)$,
 is a sequence of vector spaces $\{H_k(X): k = 0, 1, 2, 3, \ldots\}$,
 where $H_k(X)$ is called the $k$-dimensional homology of $X$. The
 dimension of $H_k(X)$, called the $k$-th Betti number of $X$, is a
 coarse measurement of the number of different holes in the space $X$
 that can be sensed by using subcomplexes of dimension $k$.

For example, the dimension of $H_0(X)$ is equal to the number of
connected components of $X$. These are the types of features
(holes) in $X$ that can be detected by using points and edges--
with this construction one is answering the question: are two points
connected by a sequence of edges or not? The simplest basis for
$H_0(X)$ consists of a choice of vertices in $X$, one in each
path-component of $X$. Likewise, the simplest basis for $H_1(X)$
consists of loops in $X$, each of which surrounds a
hole in $X$. For example, if $X$ is a graph, then the space
$H_1(X)$ encodes the number and types of cycles in the graph, this
space has the structure of a vector space.  Let $X$ denote a
simplicial complex. Define for each $k \geq 0$, the vector space
$C_k(X)$ to be the vector space whose basis is the set of oriented
$k$-simplices of $X$; that is, a $k$-simplex $\{v_0, \ldots , v_k\}$
together with an order type denoted $[v_0, \ldots , v_k]$ where a
change in orientation corresponds to a change in the sign of the
coefficient: $[v_0, \ldots , v_i, \ldots , v_j , \ldots , v_k] =
-[v_0, \ldots , v_j , \ldots , v_i, \ldots , v_k]$ if odd permutation
is used.

For $k$ larger than the dimension of $X$, we set $C_k(X) = 0$. The
boundary map is defined to be the linear transformation $\partial :
C_k \rightarrow C_{k-1}$ which acts on basis elements $[v_0, \ldots , v_k]$ via

\begin{equation} \label{action}
\partial[v_0, \ldots , v_k] :=
\sum_{i=0}^k (-1)^i[v_0, \ldots, v_{i-1}, v_{i+1}, \ldots , v_k].
\end{equation}
This gives rise to a chain complex: a sequence of vector spaces and linear transformations

$$\cdots \stackrel{\partial}{\rightarrow} C_{k+1} \stackrel{\partial}{\rightarrow}  C_k \stackrel{\partial}{\rightarrow}  C_{k-1} \cdots \stackrel{\partial}{\rightarrow}  C_2 \stackrel{\partial}{\rightarrow}  C_1 \stackrel{\partial}{\rightarrow}  C_0$$

Consider the following two subspaces of $C_k$: the cycles
(those subcomplexes without boundary) and the boundaries
(those subcomplexes which are themselves boundaries) formally defined
as:

\begin{itemize}

	\item $k-\mbox{cycles}$: $Z_k(X) =
	\mbox{ker}(\partial:C_k\rightarrow C_{k-1})$

	\item $k-\mbox{boundaries}$: $B_k(X) =
	\mbox{im}(\partial:C_{k+1}\rightarrow C_{k})$

\end{itemize}

A simple lemma demonstrates that $\partial\circ\partial = 0$; that is,
the boundary of a chain has empty boundary. It follows that $B_k$ is
a subspace of $Z_k$. This has great implications.  The $k$-cycles in $X$
are the basic objects which count the presence of a ``hole of dimension
k'' in $X$. But, certainly, many of the $k$-cycles in $X$ are measuring the
same hole; still other cycles do not really detect a hole at all --
they bound a subcomplex of dimension $k + 1$ in $X$.  We say that two
cycles $\zeta$ and $\eta$ in $Z_k(X)$ are homologous if their difference is a
boundary:

$$[\zeta]=[\eta]\,\,\,\leftrightarrow\,\,\, \zeta-\eta\in B_k(X).$$

The $k$-dimensional homology of $X$, denoted $H_k(X)$ is the quotient vector space

\begin{equation}\label{defhomology}
H_k(X):=\frac{Z_k(X)}{B_k(X)}.
\end{equation}

Specifically, an element of $H_k(X)$ is an equivalence class of
homologous $k$-cycles. This inherits the structure of a vector space
in the natural way $[\zeta]+[\eta] = [\zeta+\eta]$ and
$c[\zeta]=[c\zeta]$.

A map $f : X \rightarrow Y$ is a homotopy equivalence if
there is a map $g : Y \rightarrow X$ so that $f\circ g$ is homotopic
to the identity map on $Y$ and $g\circ f$ is homotopic to the identity
map on $X$. This notion is a weakening of the notion of
homeomorphism, which requires the existence of a continuous map
$g$ so that $f\circ g$ and $g\circ f$ are equal to the corresponding
identity maps. The less restrictive notion of homotopy equivalence is
useful in understanding relationships between complicated spaces and
spaces with simple descriptions. We say two spaces $X$ and $Y$ are
homotopy equivalent, or have the same homotopy type
if there is a homotopy equivalence from $X$ to $Y$ . This is denoted
by $X\sim Y$.

By arguments utilizing barycentric subdivision, one may show that the
homology $H_\ast(X)$ is a topological invariant of $X$: it is
indeed an invariant of homotopy type.  Readers familiar with the Euler
characteristic of a triangulated surface will not find it odd that
intelligent counting of simplices yields an invariant.  For a simple
example, the reader is encouraged to contemplate the ``physical''
meaning of $H_1(X)$. Elements of $H_1(X)$ are equivalence classes of
(finite collections of) oriented cycles in the $1$-skeleton of $X$,
the equivalence relation being determined by the $2$-skeleton of X.

Is it often remarked that homology is functorial, by which it is meant
that things behave the way they ought. A simple example of this which
is crucial to our applications arises as follows.  Consider two
simplicial complexes $X$ and $X'$. Let $f : X \rightarrow X'$ be a
continuous simplicial map: $f$ takes each $k$-simplex of $X$ to a
$k'$-simplex of $X'$, where $k'\leq k$.  Then, the map $f$ induces a
linear transformation $f_\# : C_k(X) \rightarrow C_k(X')$. It is a
simple lemma to show that $f_\#$ takes cycles to cycles and boundaries to
boundaries; hence there is a well-defined linear transformation on the
quotient spaces
$$f_\ast : H_k(X) \rightarrow H_k(X'),\,\, f_\ast([\zeta]) = [f_\#(\zeta)]. $$

 This is called the induced homomorphism of $f$ on
 $H_\ast$. Functoriality means that (1) if $f:X\rightarrow Y$ is
 continuous then $f_\ast:H_k(X)\rightarrow H_k(Y)$ is a group homomorphism;
 and (2) the composition of two maps $g \circ f$ induces the
 composition of the linear transformation: $(g \circ f)_\ast = g_\ast
 \circ f_\ast$.

\section{Background on Geometry}
\label{sec:app-2}
The development of Morse theory has been instrumental in classifying manifolds and represents a pathway between geometry and topology.  A classic reference is Milnor \cite{milnor}.

For some smooth $f:\M \to \R$, consider a point $p \in \M$ where in local coordinates the derivative vanishes, $\partial f /\partial x_1 =0, \ldots , \partial f /\partial x_d =0$.  Then that point is called a critical point, and the evaluation $f(p)$ is called a critical value.  A critical point $p \in \M$ is called non-degenerate if the Hessian $(\partial^2 f/\partial_i \partial_j)$ is nonsingular.  Such functions are called Morse functions.

Since the Hessian at a critical point is nondegenerate, there will be a
mixture of positive and negative eigenvalues.  Let $\eta$ be the number of
negative eigenvalues of the Hessian at a critical point called the
Morse index.
The basic Morse lemma states that at a critical point $p \in \M$ with index
$\eta$ and some neighborhood $\mathcal{U}$ of $p$, there exists local
coordinates $x=(x_1, \ldots , x_d)$ so that $x(p)=0$ and
\[
f(q)=f(p)-x_1(q)^2 - \cdots -x_\eta(q)^2 + x_{\eta + 1}(q)^2+ \cdots x_d(q)^2
\]
for all $q \in \mathcal{U}$.

Based on this result one is able to show that at a critical point $p \in \M$, with $f(p) = a$ say, that the sublevel set $\M_{f \leq a}$ has the same homotopy type as that of the sublevel set $\M_{f \leq a - \varepsilon}$ (for some small $\varepsilon > 0$) with an $\eta$-dimensional cell attached to it.  In fact, for a compact $\M$, its homotopy type is that of a cell complex with one $\eta$-dimensional cell for each critical point of index $\eta$.  This cell complex is known as a CW complex in homotopy theory, if the cells are attached in the order
of their dimension.

The famous set of Morse inequalities states that if $\beta_k$ is the $k-$th Betti
number and $m_k$ is the number of critical points of index $k$, then
\begin{eqnarray*}
\beta_0 &\leq& m_0 \\
\beta_1-\beta_0 &\leq& m_1-m_0 \\
\beta_2-\beta_1+\beta_0 &\leq& m_2-m_1+m_0 \\
&\cdots& \\
\chi (M) = \sum_{k=0}^d(-1)^k\beta_k &=& \sum_{k=0}^d(-1)^km_k
\end{eqnarray*}
where $\chi$ denotes the Euler characteristic.

\baselineskip = 10pt plus 3pt minus 3pt

\end{document}